\documentclass{amsart}
\usepackage{amsmath}
\usepackage{amscd}
\usepackage{amssymb}
\usepackage{amsthm}
\RequirePackage{filecontents}
\usepackage{fullpage}
\usepackage{longtable}
\usepackage[numbers]{natbib}  
\usepackage[draft]{hyperref}

\usepackage{tikz}
\usetikzlibrary{arrows}
\usetikzlibrary{positioning}
\usepackage[
top = 2.5cm,
bottom = 2.5cm,
left = 3cm,
right = 3cm]{geometry}

\theoremstyle{plain}
\newtheorem{prop}{Proposition}

\theoremstyle{definition}
\newtheorem{defn}[prop]{Definition}
\newtheorem{rem}[prop]{Remark}
\newtheorem{ex}[prop]{Example}

\title{Remarks on the theta decomposition of vector-valued Jacobi forms}
\author{Brandon Williams }

\subjclass[2010]{11F27,11F50}
\address{Fachbereich Mathematik \\ Technische Universit\"at Darmstadt \\ 64289 Darmstadt, Germany}

\email{bwilliams@mathematik.tu-darmstadt.de}

\begin{document}

\nocite{*}

\maketitle

\begin{abstract} We study the theta decomposition of Jacobi forms of nonintegral lattice index for a representation that arises in the theory of Weil representations associated to even lattices, and suggest possible applications. \end{abstract}

\section{Introduction}

This note is concerned with modular forms for the (dual) Weil representation $\rho^*$ attached to an even lattice $\Lambda$ with quadratic form $Q$. Letting $A = \Lambda'/\Lambda$ denote the discriminant group of $\Lambda$, this is a representation of the metaplectic group $Mp_2(\mathbb{Z})$ on $\mathbb{C}[A]$. Weil representations are a family of representations with better-than-usual arithmetic properties; for one thing, there are bases of modular forms in every weight with integral coefficients \cite{McG}. (Inducing the trivial representation from noncongruence subgroups to $SL_2(\mathbb{Z})$ gives plenty of examples of representations for which this fails.) These representations are important in the context of theta lifts such as \cite{B}. \par \

When $\Lambda$ is positive-definite it is well-known that modular forms for $\rho$ are essentially the same as Jacobi forms of lattice index $\Lambda$. This is a higher-dimensional version of the theta decomposition of chapter 5 of \cite{EZ}. One immediate advantage of Jacobi forms is their natural algebraic structure; multiplying Jacobi forms corresponds to an unintuitive operation on vector-valued modular forms. \par \

In the vector-valued setting, many objects that are usually indexed by integers $n \in \mathbb{Z}$ have to be replaced by pairs $(n,\gamma)$ with $\gamma \in A$ and $n \in \mathbb{Z} - Q(\gamma)$ (as in the final paragraph of the introduction of \cite{Br}.) In particular we might consider Jacobi forms of index $(m,\beta)$ where $\beta \in A$ and $m \in \mathbb{Z} - Q(\beta)$ is not generally integral. Such Jacobi forms of rational index have appeared in \cite{W}. \par \

The general theory of Jacobi forms indexed by a rational, symmetric matrix was worked out by Raum \cite{R}. Nonintegral indices force the Heisenberg group to act through a nontrivial representation. The irreducible representations that can arise were classified in \cite{R}, where theta decompositions that reduce all such Jacobi forms to vector-valued modular forms also appear. \par \

In the first half of this note we define a family of (generally not irreducible) representations $\sigma_B^*$ of the Heisenberg group that appear when one considers the Fourier-Jacobi expansion of Siegel theta functions. For these Jacobi forms the theta decomposition is particularly explicit and has a simpler proof; it is little more than applying an index-raising Hecke operator to reduce to the theta decomposition of integral-index Jacobi forms. \par \

In the second half of this note we suggest some applications. We remark that Eisenstein series and Poincar\'e series match up under this isomorphism (proposition 8) so computing bases of rather general Jacobi forms reduces to Bruinier's formulas (\cite{Br}, section 1.2). We point out that setting $z=0$ in the resulting Jacobi forms leads to some isomorphisms between modular forms for lattices of unrelated discriminants and work out an amusing example (example 11). Finally we give another interpretation of quasi-pullbacks of Borcherds products (as in \cite{M}) and construct some Borcherds products of small weight on lattices of signature $(4,2)$. \par

\section{Jacobi forms of lattice index}

For simplicity we will work with fixed Gram matrices, rather than abstract discriminant forms or even lattices. Therefore let $\mathbf{S}$ be a Gram matrix (that is, a nondegenerate, symmetric integer matrix with even diagonal) of size $e \in \mathbb{N}$ and let $\rho^* = \rho^*_{\mathbf{S}}$ denote the (dual) Weil representation of the lattice $\mathbb{Z}^{e}$ with quadratic form $Q(x) = \frac{1}{2}x^T \mathbf{S}x$ and bilinear form $\langle x,y \rangle = Q(x+y) - Q(x) - Q(y)$ (as in \cite{B},\cite{Br}). If $A = A(\mathbf{S})$ denotes the discriminant group $\mathbf{S}^{-1} \mathbb{Z}^{e} / \mathbb{Z}^{e}$ and $\mathfrak{e}_{\gamma}$, $\gamma \in A$ denotes the natural basis of the group ring $\mathbb{C}[A]$ then on the usual generators $$S = \Big( \begin{pmatrix} 0 & -1 \\ 1 & 0 \end{pmatrix}, \sqrt{\tau} \Big), \quad T = \Big( \begin{pmatrix} 1 & 1 \\ 0 & 1 \end{pmatrix}, 1 \Big)$$ of $Mp_2(\mathbb{Z})$, $\rho^*$ is given by \begin{align*} &\rho^*(S)\mathfrak{e}_{\gamma} = \frac{\mathbf{e}(\mathrm{sig}(\mathbf{S}) / 8)}{\sqrt{|A|}} \sum_{\beta \in A} \mathbf{e}\Big( \langle \gamma, \beta \rangle \Big) \mathfrak{e}_{\beta}, \\ &\rho^*(T) \mathfrak{e}_{\gamma} = \mathbf{e}\Big( - Q(\gamma) \Big) \mathfrak{e}_{\gamma}.\end{align*} Here $\mathbf{e}(\tau) = e^{2\pi i \tau}$ and $\mathrm{sig}(\mathbf{S}) = b^+ - b^-$ if $b^+,b^-$ are the numbers of positive and negative eigenvalues of $\mathbf{S}$. \par \

For $N \in \mathbb{N}$, let $\mathcal{H}_{N}$ denote the Heisenberg group of degree $N$: $$\mathcal{H}_{N} = \Big\{ (\lambda,\mu,t) \in \mathbb{Z}^{N} \times \mathbb{Z}^{N} \times \mathbb{Z}^{N \times N}: \; t - \lambda \mu^T \; \text{symmetric} \Big\}$$ with group operation $$(\lambda_1,\mu_1,t_1) \cdot (\lambda_2,\mu_2,t_2) = (\lambda_1 + \lambda_2,\mu_1 + \mu_2, t_1 + t_2 + \lambda_1\mu_2^T - \mu_1 \lambda_2^T).$$ For any matrix $B \in \mathbb{Q}^{e \times N}$ for which $\mathbf{S}B$ is integral, we can define a unitary representation $\sigma_B^*$ of $\mathcal{H}_{N}$ on $\mathbb{C}[A]$ by $$\sigma_B^*(\lambda,\mu,t) \mathfrak{e}_{\gamma} = \mathbf{e}\Big( -\gamma^T \mathbf{S} B \mu - \frac{1}{2} \mathrm{tr}(B^T \mathbf{S} B (t - \lambda \mu^T)) \Big) \mathfrak{e}_{\gamma - B \lambda},$$ and this depends only on $B$ mod $\mathbb{Z}^{e \times N}$. (When $N = 1$ this reduces to the representations $\sigma_{\beta}^*$ of \cite{W}.) These representations are compatible with $\rho^*$ in the sense that $$\rho^*(M)^{-1} \sigma_B^*(\zeta) \rho^*(M) = \sigma_B^*(\zeta \cdot M),$$ where $Mp_2(\mathbb{Z})$ acts on $\mathcal{H}_{N}$ from the right by $$(\lambda,\mu,t) \cdot \begin{pmatrix} a & b \\ c & d \end{pmatrix} = (a \lambda + c \mu, b \lambda + d \mu, t),$$ so they combine to a representation $\rho_B^*$ of the degree $N$ Jacobi group: $$\rho_B^* : \mathcal{J}_{N} = \mathcal{H}_{N} \rtimes Mp_2(\mathbb{Z}) \longrightarrow \mathrm{GL}\, \mathbb{C}[A],$$ $$\rho_B^*(\lambda,\mu,t,M) = \rho^*(M) \sigma_B^*(\lambda,\mu,t).$$ All of this can be shown directly but it is easier to prove by restricting the Weil representation of Siegel modular groups to their parabolic subgroups. \par \

\begin{defn} A \textbf{modular form} of weight $k \in \frac{1}{2}\mathbb{N}_0$ for the Gram matrix $\mathbf{S}$ is a holomorphic function $f : \mathbb{H} \rightarrow \mathbb{C}[A]$ that satisfies the functional equations $$f\Big( \frac{a \tau + b}{c \tau + d} \Big) = (c \tau + d)^k \rho^*(M) f(\tau), \quad M = \Big( \begin{pmatrix} a & b \\ c & d \end{pmatrix}, \sqrt{c\tau + d} \Big) \in Mp_2(\mathbb{Z})$$ and is bounded in $\infty$. The space of modular forms of weight $k$ will be denoted $M_k(\mathbf{S})$.
\end{defn}

Any modular form can be expanded as a Fourier series $f(\tau) = \sum_{\gamma \in A(\mathbf{S})} \sum_{n \in \mathbb{Z} - Q(\gamma)} c(n,\gamma) q^n \mathfrak{e}_{\gamma}$ where $c(n,\gamma) = 0$ unless $n \ge 0$. It follows from the functional equation under $Z = (-I,i)$ that $M_k(\mathbf{S}) = 0$ unless $k + \mathrm{sig}(\mathbf{S})/2$ is an integer. Moreover the coefficients $c(n,\gamma)$ satisfy $$c(n,-\gamma) = (-1)^{k + \mathrm{sig}(\mathbf{S})/2} c(n,\gamma)$$ so we refer to $k$ as a \textbf{symmetric} or \textbf{antisymmetric weight} as $k + \mathrm{sig}(\mathbf{S})/2$ is even or odd, respectively. \par \

\begin{defn} Let $\mathcal{M} \in \mathbb{Q}^{N \times N}$ be a symmetric positive-definite matrix and let $B$ be a matrix as above. A \textbf{Jacobi form} of weight $k \in \frac{1}{2} \mathbb{N}_0$ and index $(\mathcal{M},B)$ for the Gram matrix $\mathbf{S}$ is a holomorphic function $\Phi : \mathbb{H} \times \mathbb{C}^{N} \rightarrow \mathbb{C}[A]$ that satisfies the functional equations $$\Phi \Big( \frac{a \tau + b}{c \tau + d}, \frac{z}{c \tau + d} \Big) = (c \tau + d)^k \mathbf{e}\Big( \frac{c}{c \tau + d} z^T \mathcal{M} z \Big) \rho^*(M) \Phi(\tau,z), \quad M = \Big( \begin{pmatrix} a & b \\ c & d \end{pmatrix}, \sqrt{c \tau + d} \Big) \in Mp_2(\mathbb{Z})$$ and $$\Phi(\tau, z + \lambda \tau + \mu) = \mathbf{e}\Big(-\tau \lambda^T \mathcal{M} \lambda - \mathrm{tr}\, \mathcal{M}(2 \lambda z^T + t + \mu \lambda^T) \Big) \sigma_B^*(\lambda,\mu,t) \Phi(\tau,z)$$ for all $(\lambda,\mu,t)$ in the Heisenberg group (in particular, the left side of the above equation should be independent of $t$) and for which the usual vanishing condition on Fourier coefficients holds (see below). We denote the space of Jacobi forms as above by $J_{k,\mathcal{M},B}(\mathbf{S})$.
\end{defn}

This transformation law is often expressed through the slash operator $|_{k,\mathcal{M},B}$: $$\Phi \Big|_{k,\mathcal{M},B} M(\tau,z) = (c \tau + d)^{-k} \mathbf{e}\Big( -\frac{c}{c \tau + d} z^T \mathcal{M} z \Big) \rho^*(M)^{-1} \Phi(M \cdot \tau,z)$$ and $$\Phi \Big|_{k,\mathcal{M},B} (\lambda,\mu,t)(\tau,z) = \mathbf{e}\Big( \tau \lambda^T \mathcal{M} \lambda + \mathrm{tr}\, \mathcal{M}(2 \lambda z^T + t + \mu \lambda^T) \Big) \sigma_B^*(\lambda,\mu,t)^{-1} \Phi(\tau,z + \lambda \tau + \mu).$$

Jacobi forms have Fourier expansions in both variables; we write this out as $$\Phi(\tau,z) = \sum_{\gamma \in A(\mathbf{S})} \sum_{n \in \mathbb{Q}} \sum_{r \in \mathbb{Q}^{N}} c(n,r,\gamma) q^n \zeta^r \mathfrak{e}_{\gamma},$$ where $q^n = e^{2\pi i n \tau}$ and $\zeta^r = e^{2\pi i r^T z}$. The vanishing condition referred to above is then $c(n,r,\gamma) = 0$ whenever $4n - r^T \mathcal{M}^{-1} r < 0$. The indices $(n,r,\gamma)$ and $(\mathcal{M},B)$ that can actually occur are also restricted as follows: \par \

\noindent (i) the transformation under $T \in Mp_2(\mathbb{Z})$ implies $c(n,r,\gamma) = 0$ unless $n \in \mathbb{Z} - Q(\gamma)$; \\ (ii) the transformation under $(0,0,e_i e_j^T + e_j e_i^T) \in \mathcal{H}_{N}$ implies $c(n,r,\gamma) = 0$ unless $\mathcal{M} + \frac{1}{2}B^T \mathbf{S}B$ is half-integral, with integral diagonal; \\ (iii) the transformation under $(0,\mu,0) \in \mathcal{H}_{N}$ implies $c(n,r,\gamma) = 0$ unless $r \in \mathbb{Z}^{N} - B^T \mathbf{S} \gamma$. \par \

Moreover, the familiar restriction on weights holds: $J_{k,\mathcal{M},B}(\mathbf{S}) = 0$ unless $k + \mathrm{sig}(\mathbf{S})/2$ is integral, in which case $c(n,-r,-\gamma) = (-1)^{k + \mathrm{sig}(\mathbf{S})/2} c(n,r,\gamma)$ for all Jacobi forms $\sum c(n,r,\gamma) q^n \zeta^r \mathfrak{e}_{\gamma}$. \par \

\begin{rem} Vector-valued modular forms and Jacobi forms for a fixed Gram matrix $\mathbf{S}$ do not have natural ring structures although they are modules over the ring of scalar-valued modular forms $M_*$ in a natural way. However we will make reference to the operation of tensoring at several points. If $\mathbf{S}_1$ and $\mathbf{S}_2$ are two Gram matrices (possibly of different sizes) then tensoring their Weil representations gives $$\rho^*_{\mathbf{S}_1} \otimes \rho^*_{\mathbf{S}_2} \cong \rho^*_{\mathbf{S}_1 \oplus \mathbf{S}_2}, \quad \mathbf{S}_1 \oplus \mathbf{S}_2 = \begin{pmatrix} \mathbf{S}_1 & 0 \\ 0 & \mathbf{S}_2 \end{pmatrix}$$ and therefore the tensor product defines a bilinear map $$\otimes : M_{k_1}(\mathbf{S}_1) \times M_{k_2}(\mathbf{S}_2) \longrightarrow M_{k_1 + k_2}(\mathbf{S}_1 \oplus \mathbf{S}_2)$$ for all valid weights $k_1,k_2$. Similarly if $(\mathcal{M}_1,B_1)$ and $(\mathcal{M}_2,B_2)$ are valid indices for Jacobi forms with the same number of elliptic variables (i.e. the same $N$) then it follows from $$\sigma_{B_1}^* \otimes \sigma_{B_2}^* \cong \sigma_B^*, \quad B = \begin{pmatrix} B_1 \\ B_2 \end{pmatrix}$$ that the tensor product defines a bilinear map $$\otimes : J_{k_1,\mathcal{M}_1,B_1}(\mathbf{S}_1) \times J_{k_2,\mathcal{M}_2,B_2}(\mathbf{S}_2) \longrightarrow J_{k_1+k_2,\mathcal{M}_1+\mathcal{M}_2,B}(\mathbf{S}_1 \oplus \mathbf{S}_2).$$
\end{rem}

\begin{rem} We recover the familiar definition of Jacobi forms of lattice index (e.g. \cite{G}) by restricting to the case that $B=0$ and $\mathrm{det}(\mathbf{S}) = 1$ and by letting $2\mathcal{M}$ be the Gram matrix of the index lattice with respect to any basis.
\end{rem}

\section{Hecke $U$-operators}

The Hecke operators $U_{\ell}$, $\ell \in \mathbb{N}$ of Eichler and Zagier \cite{EZ} generalize immediately to Jacobi forms of lattice index: we define $$U_{\ell} : J_{k,\mathcal{M},B}(\mathbf{S}) \rightarrow J_{k,\ell^2 \mathcal{M}, \ell B}(\mathbf{S}), \quad U_{\ell}\Phi(\tau,z) = \Phi(\tau,\ell z).$$ The map $U_{\ell}$ is injective and $M_*$-linear and its image of $U_{\ell}$ consists of those Jacobi forms whose Fourier coefficients $c(n,r,\gamma)$ vanish unless $\ell^{-1} r \in \mathbb{Z}^{N} - B^T \mathbf{S} \gamma.$ \par \

There is a corresponding operator $U_{\ell}$ that acts on modular forms for the Weil representation attached to a (not necessarily direct) sum of two quadratic forms. Let $\mathbf{S}_1 \in \mathbb{Z}^{e \times e}, \mathbf{S}_2 \in \mathbb{Z}^{N \times N}$ be Gram matrices and let $U \in \mathbb{Z}^{e \times N}$. We define $$U_{\ell} = U_{\ell}^{e,N} : M_k \Big( \begin{pmatrix} \mathbf{S}_1 & U \\ U^T & \mathbf{S}_2 \end{pmatrix} \Big) \longrightarrow M_k \Big( \begin{pmatrix} \mathbf{S}_1 & \ell U \\ \ell U^T & \ell^2 \mathbf{S}_2 \end{pmatrix} \Big)$$ as the pullback $\uparrow_{\mathbb{Z}^{e} \times \mathbb{Z}^{N}}^{\mathbb{Z}^{e} \times \ell \mathbb{Z}^{N}}$ in the sense of lemma 5.5 of \cite{Br}; that is, $$U_{\ell} \Big[ \sum_{n,\gamma} c(n,\gamma) q^n \mathfrak{e}_{\gamma} \Big] = \sum_{n,\gamma} c(n,\gamma) q^n \times \sum_{p(\delta) = \gamma} \mathfrak{e}_{\delta},$$ where $p$ is a projection map that amounts to multiplying the second block component by $\ell$. \par \

\begin{ex} The weight three Eisenstein series for the Weil representation attached to the lattice with Gram matrix $\begin{pmatrix} 2 & 1 \\ 1 & 2 \end{pmatrix}$ is $$E_3(\tau) = (1 + 72q + 270q^2 + 720q^3 + ...) \mathfrak{e}_{(0,0)} + (27q^{2/3} + 216q^{5/3} + 459q^{8/3} +...) (\mathfrak{e}_{(1/3,1/3)} + \mathfrak{e}_{(2/3,2/3)}).$$ Letting $e = N = 1$, its image under the Hecke map $U_2$ is \begin{align*} U_2 E_3(\tau) &= (1 + 72q + 270q^2 + ...) (\mathfrak{e}_{(0,0)} + \mathfrak{e}_{(0,1/2)}) + \\ &\quad +  (27q^{2/3} + 216q^{5/3} + 459q^{8/3} + ...) (\mathfrak{e}_{(1/3,2/3)} + \mathfrak{e}_{(2/3,1/3)} + \mathfrak{e}_{(1/3,1/6)} + \mathfrak{e}_{(2/3,5/6)}), \end{align*} which is a weight three modular form for the Weil representation attached to the Gram matrix $\begin{pmatrix} 2 & 2 \\ 2 & 8 \end{pmatrix}$.
\end{ex}

\begin{rem} One can define $U_{\ell}$ on arbitrary (holomorphic) functions $f(\tau,z)$ by $U_{\ell} f(\tau,z)= f(\tau,\ell z)$. Then $U_{\ell}$ respects the Petersson slash operator in the sense that $$U_{\ell} \Big( f|_{k,\mathcal{M},B} M \Big) = (U_{\ell} f) \Big|_{k,\ell^2 \mathcal{M}, \ell B} M$$ and $$U_{\ell} \Big( f |_{k,\mathcal{M},B} (\ell \lambda, \ell \mu, \ell^2 t) \Big) = (U_{\ell} f) \Big|_{k,\ell^2 \mathcal{M}, \ell B} (\lambda, \mu, t).$$ A similar statement holds for the operator $U_{\ell}$ on modular forms.
\end{rem}

\section{Theta decomposition}

Let $\mathbf{S}_1$ and $\mathbf{S}_2$ be Gram matrices of sizes $e,N$ where $\mathbf{S}_2$ is positive-definite and set $\mathcal{M} = \frac{1}{2}\mathbf{S}_2$. There is an isomorphism of modules $$M_{k-N/2}(\mathbf{S}_1 \oplus \mathbf{S}_2) \longrightarrow J_{k,\mathcal{M},0}(\mathbf{S}_1)$$ over the ring $M_*$ of scalar modular forms which is essentially given by tensoring with the theta function $\Theta_{\mathbf{S}_2}(\tau,z)$ of $\mathbf{S}_2$ and applying a trace map (see the next section for more details) and which identifies the modular form $$F(\tau) = \sum_{\gamma \in A(\mathbf{S}_1 \oplus \mathbf{S}_2)} \sum_n c(n,\gamma) q^n \mathfrak{e}_{\gamma}$$ with the Jacobi form $$\Phi(\tau,z) = \Theta(F)(\tau,z) = \sum_{\gamma \in A(\mathbf{S}_1)} \sum_{n,r} c(n,(\gamma,\mathbf{S}_2^{-1} r)) q^{n + r^T \mathbf{S}_2^{-1} r / 2} \zeta^r \mathfrak{e}_{\gamma}.$$


In this section we record a simple generalization of this. Let $\mathbf{S}$ be a Gram matrix of size $e$ and let $B \in \mathbb{Q}^{e \times N}$ be a matrix for which $\mathbf{S}B$ is integral. Also let $\mathcal{M} \in \mathbb{Q}^{N \times N}$ be a symmetric positive-definite matrix for which $2\mathcal{M} + B^T \mathbf{S}B$ is a Gram matrix, i.e. integral with even diagonal, and set $$\mathbf{\tilde S} = \begin{pmatrix} \mathbf{S} & \mathbf{S} B \\ B^T \mathbf{S} & 2 \mathcal{M} + B^T \mathbf{S} B \end{pmatrix}.$$

\begin{prop} There is an isomorphism of $M_*$-modules $$\Theta : M_{*-N/2}(\mathbf{\tilde S}) \longrightarrow J_{*,\mathcal{M},B}(\mathbf{S})$$ that sends a modular form $$F(\tau) = \sum_{\gamma \in A(\mathbf{\tilde S})} \sum_{n \in \mathbb{Z}^{e + N} - \tilde Q(\gamma)} c(n,\gamma) q^n \mathfrak{e}_{\gamma}$$ to the Jacobi form $$\Phi(\tau,z) = \sum_{\gamma \in A(\mathbf{S})} \sum_{n \in \mathbb{Z} - Q(\gamma)} \sum_{r \in \mathbb{Z}^{N} - B^T \mathbf{S}\gamma} c\Big(n, (\gamma - \frac{1}{2} B \mathcal{M}^{-1} r, \frac{1}{2} \mathcal{M}^{-1} r)\Big) q^{n + \frac{1}{4} r^T \mathcal{M}^{-1} r} \zeta^r \mathfrak{e}_{\gamma}.$$
\end{prop}
\begin{proof} First we should clarify that if $\gamma \in \mathbf{S}^{-1} \mathbb{Z}^e$ and $r \in \mathbb{Z}^N - B^T \mathbf{S} \gamma$, then $$\begin{pmatrix} \mathbf{S} & \mathbf{S} B \\ B^T \mathbf{S} & 2 \mathcal{M} + B^T \mathbf{S} B \end{pmatrix} \begin{pmatrix} \gamma - \frac{1}{2} B \mathcal{M}^{-1} r \\ \frac{1}{2} \mathcal{M}^{-1} r \end{pmatrix} = \begin{pmatrix} \mathbf{S} \gamma \\ B^T \mathbf{S} \gamma + r \end{pmatrix} \in \mathbb{Z}^{e+N}$$ and therefore $$\left(\gamma - \frac{1}{2}B \mathcal{M}^{-1} r, \frac{1}{2} \mathcal{M}^{-1} r\right) \in \mathbf{\tilde S}^{-1} \mathbb{Z}^{e+N};$$ moreover this is a well-defined element of $A(\mathbf{\tilde S})$ when $\gamma$ is interpreted as an element of $A(\mathbf{S})$. \par

Choose $\ell \in \mathbb{N}$ such that $\ell \cdot B \in \mathbb{Z}^{e \times N}$. The Gram matrix $\begin{pmatrix} \mathbf{S} & \ell \mathbf{S} B \\ \ell B^T \mathbf{S} & \ell^2 (2 \mathcal{M} + B^T \mathbf{S} B) \end{pmatrix}$ is equivalent to $\mathbf{S} \oplus 2 \ell^2 \mathcal{M}$ so for every weight $k$ we have an injective map $$M_{k-N / 2}(\mathbf{\tilde S}) \stackrel{U_{\ell}}{\longrightarrow} M_{k-N / 2}(\mathbf{S} \oplus 2 \ell^2 \mathcal{M})\longrightarrow J_{k,\ell^2 \mathcal{M},0}(\mathbf{S}).$$ In terms of Fourier expansions this sends $F(\tau) = \sum_{(\gamma_1,\gamma_2) \in A(\mathbf{\tilde S})} \sum_n c(n,(\gamma_1,\gamma_2)) q^n \mathfrak{e}_{(\gamma_1,\gamma_2)}$ first to the modular form $$U_{\ell} F(\tau) = \sum_{\gamma_1 \in A(\mathbf{S})} \sum_{\gamma_2 \in A(2 \ell^2 \mathcal{M})} \sum_n c(n,(\gamma_1 - \ell B\gamma_2, \ell \gamma_2)) q^n \mathfrak{e}_{(\gamma_1, \gamma_2)}$$ and then to the Jacobi form $$\Phi(\tau,z) = \sum_{\gamma \in A(\mathbf{S})} \sum_{n,r}  c\Big(n,\left(\gamma - \frac{1}{2\ell} B \mathcal{M}^{-1} r , \frac{1}{2\ell} \mathcal{M}^{-1} r\right)\Big) q^{n + \frac{1}{4\ell^2} r^T \mathcal{M}^{-1} r} \zeta^r \mathfrak{e}_{\gamma}.$$ These coefficients vanish unless $$\begin{pmatrix} \mathbf{S} & \mathbf{S}B \\ B^T \mathbf{S} & 2 \mathcal{M} + B^T \mathbf{S} B \end{pmatrix} \begin{pmatrix} \gamma - \frac{1}{2\ell} B \mathcal{M}^{-1} r \\ \frac{1}{2\ell} \mathcal{M}^{-1} r \end{pmatrix} = \begin{pmatrix} \mathbf{S} \gamma \\ B^T \mathbf{S} \gamma + \ell^{-1} r \end{pmatrix}$$ is integral, or equivalently $\ell^{-1} r \in \mathbb{Z}^{N} - B^T \mathbf{S} \gamma$, which is the condition for $\Phi$ to be in the range of $U_{\ell} : J_{k,\mathcal{M},B}(\mathbf{S}) \rightarrow J_{k,\ell^2 \mathcal{M},0}(\mathbf{S})$. Therefore we define the theta decomposition by requiring the diagram

\begin{center}
\begin{tikzpicture}[node distance=2cm, auto]
\node (A) {$M_{k-N/2}(\mathbf{\tilde S})$};
\node(AA) [right of = A] {$ $};
\node (B) [right of = AA] {$M_{k-N/2}(\mathbf{S} \oplus 2 \ell^2 \mathcal{M})$};
\node (C) [below of = A] {$J_{k,\mathcal{M},B}(\mathbf{S})$};
\node (D) [below of = B] {$J_{k,\ell^2 \mathcal{M},0}(\mathbf{S})$};
\draw[dashed,->] (A) to node {$ $} (C);
\draw[->] (A) to node {$U_{\ell}$} (B);
\draw[->] (C) to node {$U_{\ell}$} (D);
\draw[->] (B) to node {$ $} (D);
\end{tikzpicture}
\end{center} to commute. Its inverse can be constructed in the same way: starting with a Jacobi form, apply $U_{\ell}$ and the usual theta decomposition, then identify the result in the image of $U_{\ell}$ on modular forms. That this is independent of $\ell$ can be seen from the coefficient formula for $\Phi$ in the claim.
\end{proof}

\section{Eisenstein and Poincar\'e series}

Some of the simplest Jacobi forms that can be constructed are the Poincar\'e series: $$\mathcal{P}_{k,(\mathcal{M},B),(n,r,\gamma)}(\tau) = \sum_{(\lambda,0,0,M) \in \mathcal{J}_{N,\infty} \backslash \mathcal{J}_{N}} q^n \zeta^r \mathfrak{e}_{\gamma} \Big|_{k,\mathcal{M},B} (\zeta,M),$$ where $(\lambda,0,0,M)$ runs through cosets of the Jacobi group $\mathcal{J}_{N}$ by the subgroup generated by all $(0,\mu,t) \in \mathcal{H}_{N}$ and by $T = (\begin{pmatrix} 1 & 1 \\ 0 & 1 \end{pmatrix}, 1)$ and $Z = (-I,i)$ from $Mp_2(\mathbb{Z})$. The series $P_{n,r,\gamma}(\tau)$ is well-defined as long as $n \in \mathbb{Z} - Q(\gamma)$ and $r \in \mathbb{Z}^{N} - B^T \mathbf{S}\gamma$ and $k > 2 + N/2$ and it always transforms like a Jacobi form by construction. The vanishing condition on Fourier coefficients is satisfied when $4n - r^T \mathcal{M}^{-1} r \ge 0$. If $n=r=0$ then we call $\mathcal{E}_{k,\mathcal{M},\gamma}(\tau) = \mathcal{P}_{k,\mathcal{M},0,0,\gamma}(\tau)$ the \textbf{Eisenstein series}. See also \cite{M2} which considers the scalar-valued case in more detail. \par \

The point of this section is to show that the Jacobi Eisenstein and Poincar\'e series as above correspond to the usual Eisenstein and Poincar\'e series as in chapter $1$ of \cite{Br}: $$P_{k,n,\gamma}(\tau) = \sum_{M \in \Gamma_{\infty} \backslash \Gamma} q^n \mathfrak{e}_{\gamma} \Big|_{k,\rho^*} M;$$ i.e. any coefficient formula for Jacobi Eisenstein and Poincar\'e series even in the generality considered here (which includes the author's computations in \cite{W}) will reduce to theorem 1.3 of \cite{Br}. \par \

\begin{prop} In the notation of the previous section, the theta decomposition satisfies $$\Theta \Big( P_{k-N / 2, n, (\gamma - \frac{1}{2} B \mathcal{M}^{-1} r, \frac{1}{2} \mathcal{M}^{-1} r)} \Big) = \mathcal{P}_{k,(\mathcal{M},B),(n + r^T \mathcal{M}^{-1} r / 4, r, \gamma)}.$$
\end{prop}
Poincar\'e series extract Fourier coefficients with respect to an appropriately defined Petersson scalar product so this correspondence is suggested by the formula in proposition 7; one could give a proof along these lines but we use a more direct argument here.
\begin{proof} Suppose first that $B=0$ so the theta decomposition maps $F \in M_{k-N/2}(\mathbf{S} \oplus 2\mathcal{M})$ to $\pi(F \otimes \Theta)$ where $\Theta(\tau,z) = \sum_{\lambda \in \mathbb{Z}^{N}} q^{\lambda^T \mathcal{M}^{-1} \lambda / 4} \zeta^{\lambda} \mathfrak{e}_{(2\mathcal{M})^{-1} \lambda}$ and where $\pi$ is the linear map $$\pi : \mathbb{C}[A(\mathbf{S})] \otimes \mathbb{C}[A(2\mathcal{M})] \otimes \mathbb{C}[A(2\mathcal{M})] \rightarrow \mathbb{C}[A(\mathbf{S})], \quad \pi(\mathfrak{e}_{\beta} \otimes \mathfrak{e}_{\gamma} \otimes \mathfrak{e}_{\delta}) = \begin{cases} \mathfrak{e}_{\beta}: & \gamma = \delta; \\ 0: & \text{otherwise}. \end{cases}$$ Here $\Theta$ is a Jacobi form of weight $N / 2$ and index $\mathcal{M}$ for the (non-dual) Weil representation $\rho_{2\mathcal{M}}$ by Poisson summation. One can check that $\pi$ satisfies $\pi \circ (\rho_{\mathbf{S} \oplus 2 \mathcal{M}}^* \otimes \rho_{2\mathcal{M}}(M)) = \rho_{\mathbf{S}}^*(M) \circ \pi$ for all $M \in Mp_2(\mathbb{Z})$ on the generators $M = S,T$. It follows that if $F(\tau) = \sum_{M \in \Gamma_{\infty} \backslash \Gamma} \phi(\tau) \Big|_{k - N / 2, \rho_{\mathbf{S} \otimes 2 \mathcal{M}}^*} M$ for some seed function $\phi$ then 

\begin{align*} &\quad \sum_{M \in \Gamma_{\infty} \backslash \Gamma} \pi\Big( \phi(\tau) \otimes \Theta(\tau,z) \Big) \Big|_{k,\mathcal{M}, \rho_{\mathbf{S}}^*} M \\ &= \pi\sum_{M \in \Gamma_{\infty} \backslash \Gamma} \Big(\phi(\tau) \otimes \Theta(\tau,z)\Big) \Big|_{k,\mathcal{M}, \rho_{\mathbf{S} \oplus 2 \mathcal{M}}^* \otimes \rho_{2\mathcal{M}}} M \\ &= \pi \sum_{c,d} (c \tau + d)^{-k+N/2} \rho_{\mathbf{S} \oplus 2 \mathcal{M}}^*(M)^{-1} \phi(M \cdot \tau) \otimes (c \tau + d)^{-N/2} \mathbf{e}\Big(-\frac{c}{c\tau + d} z^T \mathcal{M} z \Big) \rho_{2\mathcal{M}}(M)^{-1} \Theta\left(\frac{a\tau + b}{c \tau + d},\frac{z}{c\tau + d}\right) \\ &= \pi \sum_{c,d} (c \tau + d)^{-k+N/2} \rho_{\mathbf{S} \oplus 2 \mathcal{M}}^*(M)^{-1} \phi(M \cdot \tau) \otimes \Theta(\tau,z) \\ &= \pi \Big(F(\tau) \otimes \Theta(\tau,z)\Big) \end{align*} is the Jacobi form corresponding to $F$. When $\phi(\tau) = q^n \mathfrak{e}_{\gamma} \otimes \mathfrak{e}_{(2 \mathcal{M})^{-1} r}$ for some $\gamma \in A(\mathbf{S})$ then 

\begin{align*} \pi\Big( \phi(\tau) \otimes \Theta(\tau,z) \Big) &= \sum_{\mu \in r + 2 \mathcal{M} \mathbb{Z}^{N}} q^{n + \mu^T \mathcal{M}^{-1} \mu / 4} \zeta^{\mu} \mathfrak{e}_{\gamma} \\ &= \sum_{\lambda \in \mathbb{Z}^{N}} \mathbf{e}\Big( \tau \lambda^T \mathcal{M} \lambda + 2\lambda^T \mathcal{M} z \Big) q^{n + r^T \mathcal{M}^{-1} r / 4 + r^T \lambda} \zeta^r \mathfrak{e}_{\gamma} \quad (\mu = r + 2 \mathcal{M} \lambda) \\ &= \sum_{\lambda \in \mathbb{Z}^{N}} q^{n + r^T \mathcal{M}^{-1} r / 4} \zeta^r \mathfrak{e}_{\gamma} \Big|_{k,\mathcal{M}} (\lambda,0,0)\end{align*} and therefore \begin{align*} \pi \Big( P_{k-N/2,n,(\gamma,(2 \mathcal{M})^{-1} r)}(\tau) \otimes \Theta(\tau,z) \Big) &= \sum_{M \in \Gamma_{\infty} \backslash \Gamma} \pi\Big( \phi(\tau) \otimes \Theta(\tau,z) \Big) \Big|_{k,\mathcal{M}, \rho_{\mathbf{S}}^*} M \\ &= \sum_{(\lambda,0,0,M) \in \mathcal{J}_{\infty} \backslash \mathcal{J}} \Big( q^{n + r^T \mathcal{M}^{-1} r / 4} \zeta^r \mathfrak{e}_{\gamma} \Big) \Big|_{k,\mathcal{M}} (\lambda,0,0) \Big|_{k,\mathcal{M},\rho_{\mathbf{S}}^*} M \\ &= \mathcal{P}_{k,\mathcal{M},(n+r^T \mathcal{M}^{-1} r / 4, r, \gamma)}(\tau,z). \end{align*}

For the general case of theta decomposition in the Gram matrix $\mathbf{\tilde S} = \begin{pmatrix} \mathbf{S} & \mathbf{S} B \\ B^T \mathbf{S} & 2 \mathcal{M} + B^T \mathbf{S} B \end{pmatrix}$ we need to check that the images of $P_{k-N/2,n,(\gamma-\frac{1}{2}B \mathcal{M}^{-1} r, \frac{1}{2} \mathcal{M}^{-1} r)}$ and $\mathcal{P}_{k,(\mathcal{M},B),(n + r^T \mathcal{M}^{-1} r / 4, r, \gamma)}$ match up correctly under $U_{\ell}$. On Jacobi forms we find \begin{align*} &\quad U_{\ell} \mathcal{P}_{k,(\mathcal{M},B),(n + r^T \mathcal{M}^{-1} r/4,r,\gamma)} \\ &= U_{\ell} \sum_{M \in \Gamma_{\infty} \backslash \Gamma} \sum_{\lambda \in \mathbb{Z}^N} (q^{n + r^T \mathcal{M}^{-1} r / 4} \zeta^r \mathfrak{e}_{\gamma}) \Big|_{k,\mathcal{M},B} (\lambda,0,0) \Big|_{k,\mathcal{M},B,\rho_{\mathbf{S}}^*} M \\ &= \sum_{M \in \Gamma_{\infty} \backslash \Gamma} \sum_{\lambda \in \mathbb{Z}^N} \Big[ U_{\ell} \Big( \sum_{\mu \in \mathbb{Z}^N / \ell \mathbb{Z}^N} q^{n + r^T \mathcal{M}^{-1} r / 4} \zeta^r \mathfrak{e}_{\gamma} \Big|_{k,\mathcal{M},B} (\mu,0,0) \Big) \Big] \Big|_{k,\ell^2\mathcal{M},\ell B} (\lambda,0,0) \Big|_{k,\ell^2 \mathcal{M},\ell B,\rho_{\mathbf{S}}^*} M \\ &= \sum_{\mu \in \mathbb{Z}^N / \ell \mathbb{Z}^N} \mathcal{P}_{k, \ell^2 \mathcal{M}, (n + (r + 2 \mathcal{M}\mu)^T \mathcal{M}^{-1}(r + 2 \mathcal{M}\mu)/4, \ell r + 2\ell \mathcal{M}\mu, \gamma + B \mu)}, \end{align*} while on modular forms, abbreviating $\tilde \gamma = (\gamma - \frac{1}{2}B \mathcal{M}^{-1} r, \frac{1}{2} \mathcal{M}^{-1} r)$ and applying $\uparrow$ termwise gives \begin{align*} U_{\ell} P_{k-N/2,n,\tilde \gamma} &= U_{\ell} \Big[ \sum_{M \in \Gamma_{\infty} \backslash \Gamma} q^n \mathfrak{e}_{\tilde \gamma} \Big|_{k,\rho_{\mathbf{\tilde S}}^*} M \Big] \\ &= \sum_{M \in \Gamma_{\infty} \backslash \Gamma} \Big( q^n \sum_{\substack{\delta \in A(\mathbf{S} \oplus 2 \ell^2 \mathcal{M}) \\ p(\delta) = \tilde \gamma}} \mathfrak{e}_{\delta} \Big) \Big|_{k,\rho_{\mathbf{S} \oplus 2 \ell^2 \mathcal{M}}^*} M \\ &= \sum_{\mu \in \mathbb{Z}^N / \ell \mathbb{Z}^N} P_{k-N/2,n,(\gamma - \frac{1}{2} \ell^{-1} \mathcal{M}^{-1} r, \frac{1}{2}\ell^{-1} \mathcal{M}^{-1} r + \ell^{-1} \mu)}, \end{align*} which correspond as they are expected to.
\end{proof}

\section{Kohnen plus space}

Since Jacobi forms of very small index tend to be determined completely by their value at $z=0$, applying the theta decomposition and setting $z=0$ in the resulting Jacobi forms gives isomorphisms between modular forms of apparently unrelated levels. The simplest example is an identification of the Kohnen plus spaces of half-integral weight modular forms of level $4$ with modular forms of level $3$ and odd integral weight. \par \

Let $k \in \mathbb{N}_0$. The Kohnen plus space of weight $k + 1/2$ is the subspace $M_{k+1/2}^+(\Gamma_0(4))$ of modular forms of level $\Gamma_0(4)$ in which the Fourier coefficient $a_n$ is zero unless $(-1)^k n \equiv 0,1$ mod $4$. This was introduced by Kohnen in \cite{K} (also for more general levels) as a natural setting for the Shimura correspondence: in particular there are Hecke-equivariant isomorphisms $M_{k+1/2}^+ \cong M_{2k}$ to classical modular forms of integral weight and level $1$. \par \

\begin{prop} For every odd integer $k$, there is an isomorphism $$M_k(\Gamma_1(3)) = M_k(\Gamma_0(3),\chi) \cong M_{k-1/2}^+(\Gamma_0(4)) \times M_{k+1/2}^+(\Gamma_0(4)) .$$ Here $\chi$ is the nontrivial Nebentypus modulo $3$. More precisely, there are isomorphisms $$M_k^+(\Gamma_0(3),\chi) \cong M_{k-1/2}^+(\Gamma_0(4)), \quad M_k^-(\Gamma_0(3),\chi) \cong M_{k+1/2}^+(\Gamma_0(4)),$$ where $M_k^+(\Gamma_0(3),\chi)$ is the subspace of forms whose coefficients $a_n$ vanish when $n \equiv 2$ mod $3$; and $M_k^-(\Gamma_0(3),\chi)$ is the subspace of forms whose coefficients vanish when $n \equiv 1$ mod $3$.
\end{prop}
These isomorphisms are not Hecke-equivariant but they are simple to express in terms of the coefficients, and more generally they are isomorphisms of graded modules over the ring $M_*$ of classical modular forms of level one; see the remark below.
\begin{proof} Using a more general result of Bruinier and Bundschuh \cite{BB}, all of the spaces here can be interpreted as spaces of vector-valued modular forms: $$M_k^+(\Gamma_0(3),\chi) \cong M_k \Big( \begin{pmatrix} -2 & -1 \\ -1 & -2 \end{pmatrix} \Big), \quad M_k^-(\Gamma_0(3),\chi) \cong M_k\Big( \begin{pmatrix} 2 & 1 \\ 1 & 2 \end{pmatrix} \Big),$$ $$M_{k-1/2}^+(\Gamma_0(4)) \cong M_{k-1/2}\Big( \begin{pmatrix} -2 & -1 &1 \\ -1 & -2 & 0 \\ 1 & 0 & 0 \end{pmatrix} \Big), \quad M_{k+1/2}^+(\Gamma_0(4)) \cong M_{k+1/2}((2)),$$ since all of the Gram matrices above have prime determinant. These identifications are essentially given by summing all components; such that for example the weight three Eisenstein series $$E_3(\tau) = (1 - 90q - 216q^2 - ...) \mathfrak{e}_{(0,0)} + (-9q^{1/3} - 117q^{4/3} - ...) (\mathfrak{e}_{(1/3,1/3)} + \mathfrak{e}_{(2/3,2/3)}) \in M_3\Big( \begin{pmatrix} -2 & -1 \\ -1 & -2 \end{pmatrix} \Big)$$ and $$E_3(\tau) = (1 + 72q + 270q^2 + ...) \mathfrak{e}_{(0,0)} + (27q^{2/3} + 216q^{5/3} + ...) (\mathfrak{e}_{(1/3,1/3)} + \mathfrak{e}_{(2/3,2/3)}) \in M_3 \Big( \begin{pmatrix} 2 & 1 \\ 1 & 2 \end{pmatrix} \Big)$$ correspond to the scalar modular forms $$1 - 18q - 90q^3 - 234q^4 - 216q^6 - ... \in M_3(\Gamma_0(3),\chi)^+$$ and $$1 + 54q^2 + 72q^3 + 432q^5 + 270q^6 + ... \in M_3(\Gamma_0(3),\chi)^-,$$ respectively. The theta decomposition gives identifications $$M_k\Big( \begin{pmatrix} 2 & 1 \\ 1 & 2 \end{pmatrix} \Big) \cong J_{k+1/2,3/4,(1/2)}((2))$$ and $$M_{k-1/2}\Big( \begin{pmatrix} -2 & -1 & 1 \\ -1 & -2 & 0 \\ 1 & 0 & 0 \end{pmatrix}\Big) \cong J_{k,1/3,(-2/3,1/3)} \Big( \begin{pmatrix} -2 & -1 \\ -1 & -2 \end{pmatrix} \Big)$$ and setting $z=0$ in the resulting Jacobi forms gives the maps $$M_k\Big( \begin{pmatrix}2 & 1 \\ 1 & 2 \end{pmatrix} \Big) \rightarrow M_{k+1/2}((2)), \; M_{k-1/2}\Big( \begin{pmatrix} -2 & -1 & 1 \\ -1 & -2 & 0 \\ 1 & 0 & 0 \end{pmatrix} \Big) \rightarrow M_k\Big( \begin{pmatrix} -2 & -1 \\ -1 & -2 \end{pmatrix} \Big)$$ in the claim. Both maps are injective: one can use theorem 1.2 of \cite{EZ} to see that any Jacobi form $\Phi$ of index strictly less than one half for which $\Phi(\tau,0)$ vanishes, or of index less than one for which $\Phi(\tau,0)$ and $\frac{d}{dz}\Big|_{z=0} \Phi(\tau,z)$ both vanish, must be identically zero. (In the second case, $\frac{d}{dz}\Big|_{z=0} \Phi(\tau,z)$ is a modular form of antisymmetric weight for the Gram matrix $(2)$ so it vanishes identically regardless of $\Phi$.) Theorem 1.2 of \cite{EZ} considers only integer-index Jacobi forms but we can reduce to that case by tensoring $\Phi$ with itself enough times to produce something of integer index. \par \

To see that these are isomorphisms we compare the dimensions of both sides. One way to show that the dimensions are the same is to check the first few weights and then use the fact that for $k \ge 3$, the dimensions in all four cases increase by $2$ when we replace $k$ by $k + 12$ (since $2$ is the number of elements $\gamma$ in each discriminant group up to equivalence $\gamma \sim -\gamma$).
\end{proof}

\begin{rem}
Unraveling this argument in terms of scalar-valued modular forms, the isomorphisms above are $$M_k^-(\Gamma_1(3)) \rightarrow M_{k+1/2}^+(\Gamma_0(4)), \quad f(\tau) \mapsto \frac{1}{18} \begin{pmatrix} f(\frac{4\tau}{3}) & f(\frac{4 \tau + 1}{3}) & f(\frac{4 \tau + 2}{3}) \end{pmatrix} \begin{pmatrix} 4 & 1 & 1 \\ 1 & 4 & 1 \\ 1 & 1 & 4 \end{pmatrix} \begin{pmatrix} \vartheta(\frac{\tau}{3}) \\ \vartheta(\frac{\tau + 1}{3}) \\ \vartheta(\frac{\tau + 2}{3}) \end{pmatrix}$$ and $$M_{k-1/2}^+(\Gamma_0(4)) \rightarrow M_k^+(\Gamma_1(3)),$$ $$f(\tau) \mapsto \frac{1}{4} \sum_{a \in \mathbb{Z}/4\mathbb{Z}} f\Big( \frac{3 \tau + a}{4} \Big) \vartheta\Big( \frac{\tau - a}{4} \Big),$$ where $\vartheta(\tau) = 1 + 2q + 2q^4 + 2q^9 + ...$ One could probably prove directly that these are well-defined isomorphisms but using the argument above this is more or less automatic.
\end{rem}

\begin{ex} The first part of remark 10 means that to compute the Kohnen plus form corresponding to $f(\tau) = \sum_n a_n q^n \in M_k^-(\Gamma_1(3))$, we divide all coefficients $a_n$, $n \not \equiv 0 \, (3)$ by two, change variables $q \mapsto q^{4/3}$, multiply by $\vartheta(\tau/3)$, and restrict to integer exponents. For example, to compute the Kohnen plus form corresponding to $1 + 54q^2 + 72q^3 + 432q^5 + 270q^6 + ...$ we multiply \begin{align*} &\quad \Big( 1 + 27q^{8/3} + 72q^4 + 216q^{20/3} + 270q^8 + ... \Big) \Big( 1 + 2q^{1/3} + 2q^{4/3} + 2q^3 + ... \Big) \\ &= 1 + 2q^{1/3} + 2q^{4/3} + 27q^{8/3} + 56q^3 +  126q^4 + 144q^{13/3} + 146q^{16/3} + 54q^{17/3} + 216q^{20/3} + 576q^7 + ... \end{align*} and restrict to integer exponents to find the Cohen Eisenstein series $$1 + 56q^3 + 126q^4 + 576q^7 + ...$$ of weight $7/2$ \cite{C}. \par \

The second part of remark 10 means that to compute the level three form corresponding to a Kohnen plus form $f(\tau)$, we multiply $f(3\tau/4)$ by $\vartheta(\tau/4)$ and restrict to integer exponents. For example, the Cohen Eisenstein series of weight $5/2$ is $$f(\tau) = 1 - 10q - 70q^4 - 48q^5 - 120q^8 - 250q^9 - 240q^{12} - 240q^{13} - ...$$ and since \begin{align*} &\quad \Big( 1 - 10q^{3/4} - 70q^3 - 48q^{15/4} - 120q^6 - 250q^{27/4} - 240q^9 - ... \Big) \Big( 1 + 2q^{1/4} + 2q + 2q^{9/4} + 2q^4 + ... \Big) \\ &= 1 + 2q^{1/4} -10q^{3/4} - 18q - 20q^{7/4} + 2q^{9/4} - 90q^3 - 140q^{13/4} - 48q^{15/4} - 234q^4 - ...\end{align*} we see that the corresponding level three form is $$1 - 18q - 90q^3 - 234q^4 - ... \in M_3^+(\Gamma_1(3)).$$ We could also point out that the theta function $\vartheta(\tau) = 1 + 2q + 2q^4 + 2q^9 + ... \in M_{1/2}^+(\Gamma_0(4))$ itself corresponds to the theta function $\theta(\tau) = 1 + 6q + 6q^3 + 6q^4 + 12q^7 + ... \in M_1^+(\Gamma_0(3),\chi)$ of $x^2 + xy + y^2$ under this isomorphism.

\end{ex}

\section{Hermitian modular products of small weight}

Recall that the Borcherds lift takes vector-valued modular forms for the (non-dual) Weil representation attached to a lattice of signature $(b^+,2)$ with singularities at cusps to orthogonal modular forms with product expansions and known divisors (\cite{B},\cite{Br}). Given such an automorphic product one can produce automorphic forms on subgrassmannians essentially by restriction, with a renormalization process to deal with a possible zero or pole there. This operation is called the quasi-pullback (see \cite{M}, section 3) and it can be used to construct infinite families of Borcherds products. For example, Gritsenko and Nikulin constructed a paramodular form of weight $12$ in every level $t$ except $t=1$ using this kind of argument for the lift of $\Delta^{-1}$ to the Grassmannian of $II_{26,2}$ (\cite{GN2}, remark 4.4). It is also possible to find automorphic forms on subgrassmannians by tensoring the original input form by various theta functions and taking the Borcherds lift of this. Ma showed in \cite{M} that these processes lead to the same result, at least when Koecher's principle applies. \par \

Theta decompositions allow one to identify the nearly-holomorphic input functions with weak Jacobi forms, and in the case of integral index, setting $z=0$ in those Jacobi forms has the same effect as tensoring with theta functions as in the last paragraph. Taking theta decompositions into Jacobi forms of nonintegral index and setting $z=0$ gives a similar interpretation of the $\Theta$-contraction of \cite{M}. \par \

We give an example here by showing that for even discriminants there always exist Hermitian modular products of small weight. These are essentially the same as orthogonal modular forms on Grassmannians of signature $(4,2)$ lattices of the form $\Lambda_K = \mathcal{O}_K \oplus II_{2,2}$, where $\mathcal{O}_K$ denotes the ring of integers in an imaginary-quadratic number field $K$ together with its norm-form $N_{K/\mathbb{Q}}$ (see \cite{Dern} for details). This is based on an observation of A. Krieg who suggested this application to me (see \cite{Kr}). \par

\begin{prop} Let $K$ be an imaginary-quadratic field of even discriminant $d_K$ and set $m = -\frac{d_K}{4}$. \\ (i) If $m$ is a sum of three nonzero squares then there is a holomorphic product for $\Lambda_K$ of weight two. \\ (ii) If $m$ is a sum of two nonzero squares then there is a holomorphic product for $\Lambda_K$ of weight three. \\ (iii) If $m=1$ then there is a holomorphic product for $\Lambda_K$ of weight four.
\end{prop}
The cases $m=1,2$ are well-known and were important in finding the structure of the rings of Hermitian modular forms for $K = \mathbb{Q}(i)$ and $K = \mathbb{Q}(\sqrt{-2})$ in \cite{DK1},\cite{DK2}. One can also find the cases $m=5,6$ in the tables of appendix B of \cite{W1}. Those tables also show that no such statement is possible for odd discriminants; for example, the holomorphic products for $\mathbb{Q}(\sqrt{-3})$ and $\mathbb{Q}(\sqrt{-7})$ of smallest weight have weight $9$ and $7$, respectively. Note that (i) and (ii) can both occur for a single $m$, e.g. $m = 17 = 1^2 + 4^2 = 2^2 + 2^2 + 3^2$, where we get holomorphic products of both weights.
\begin{proof} The products in the claim arise as quasi-pullbacks of a singular-weight product on a Grassmannian of type $O(6,2)$; we will show how these arise as Borcherds lifts using the theta decomposition. Suppose first that $m$ has a representation of the form $m = 1 + a^2 + b^2$ with $a,b \in \mathbb{N}_0$. Then the matrix $$\mathbf{S} = \begin{pmatrix} 2 & 0 & 0 & 0 \\ 0 & 2m & 2a & 2b \\ 0 & 2a & 2 & 0 \\ 0 & 2b & 0 & 2 \end{pmatrix}$$ is the Gram matrix of $x_1^2 + x_2^2 + x_3^2 + x_4^2$ with respect to the basis $(1,0,0,0),(0,1,a,b),(0,0,1,0),(0,0,0,1)$ of $\mathbb{Z}^4$. There is a nearly-holomorphic modular form $F$ of weight $-2$ with principal part $$F(\tau) = 4\mathfrak{e}_{(0,0,0,0)} + q^{-1/4} (\mathfrak{e}_{(\frac{1}{2},0,0,0)} + \mathfrak{e}_{(0,\frac{1}{2},0,0)} + \mathfrak{e}_{(0,0,\frac{1}{2},0)} + \mathfrak{e}_{(0,0,0,\frac{1}{2})}) + ...$$ (Under the Borcherds lift, this produces the singular weight product on the lattice $4A_1 \oplus II_{2,2}$. Up to an involution, i.e. taking the dual lattice and rescaling, this is the case $m=4$ of theorem 8.1 of \cite{Sch}.) Through the theta decomposition we can interpret $F$ as a weak Jacobi form $\Phi(\tau,z)$ of index $\mathcal{M} = \begin{pmatrix} \frac{1+b^2}{m} & -\frac{ab}{m} \\ -\frac{ab}{m} & \frac{1 + a^2}{m} \end{pmatrix}, \, B = \begin{pmatrix} 0 & 0 \\ a/m & b/m \end{pmatrix}$ for the Weil representation attached to the Gram matrix $\begin{pmatrix} 2 & 0 \\ 0 & 2m \end{pmatrix}$, in which we then set $z=0$. By proposition 7 the coefficient of $q^0 \mathfrak{e}_0$ in $\Phi(\tau,0)$ (i.e. twice the weight of the Borcherds lift) is $$4 + \#\left\{r \in \mathbb{Z}^2: \; \frac{1}{4}r^T \mathcal{M}^{-1} r = \frac{1}{4}\right\} = 4 + \#\{(r_1,r_2) \in \mathbb{Z}^2: \; (1+a^2) r_1^2 + 2ab r_1 r_2 + (1+b^2)r_2^2 = 1\};$$ this is $4$ if $a,b$ are both nonzero, and $6$ if one of $a,b$ is zero, and $8$ if both of $a,b$ are zero. \par \

The same argument works in the general case, but it is somewhat messier to write out: we choose a (primitive) vector $v = v_1 \in \mathbb{Z}^3$ with $v^T v = m$, extend it to a basis $\{v_1,v_2,v_3\}$ of $\mathbb{Z}^3$ and apply theta decomposition to the Gram matrix of $x_1^2 + x_2^2 + x_3^2 + x_4^2$ with respect to the basis (in block form) $\{(1,0),(0,v_1),(0,v_2),(0,v_3)\}$ of $\mathbb{Z}^4$.
\end{proof}

\begin{ex} Using proposition 7, the principal parts of these input functions can be read off the principal part of $$F(\tau) = 4\mathfrak{e}_{(0,0,0,0)} + q^{-1/4} (\mathfrak{e}_{(\frac{1}{2},0,0,0)} + \mathfrak{e}_{(0,\frac{1}{2},0,0)} + \mathfrak{e}_{(0,0,\frac{1}{2},0)} + \mathfrak{e}_{(0,0,0,\frac{1}{2})}) + ...$$  We list the first few examples below: \par \

\begin{longtable}{|l|l|l|}
\caption{Principal parts for Hermitian modular products of small weight} \\
\hline
         $m$            & Weight &              Principal part                                                          \\ \hline
$1$ & $4$ & $8 \mathfrak{e}_{(0,0)} + q^{-1/4} (\mathfrak{e}_{(1/2,0)} + \mathfrak{e}_{(0,1/2)})$ \\ \hline
$2$ & $3$ & $6 \mathfrak{e}_{(0,0)} + 2q^{-1/8} (\mathfrak{e}_{(0,1/4)} + \mathfrak{e}_{(0,3/4)}) + q^{-1/4} \mathfrak{e}_{(1/2,0)}$ \\ \hline
$5$ & $3$ & $\begin{aligned}&6 \mathfrak{e}_{(0,0)} + q^{-1/20} (\mathfrak{e}_{(0,1/10)} + \mathfrak{e}_{(0,9/10)} + \mathfrak{e}_{(1/2,2/5)} + \mathfrak{e}_{(1/2,3/5)}) + \\ &\quad + q^{-1/5} (\mathfrak{e}_{(0,1/5)} + \mathfrak{e}_{(0,4/5)}) + q^{-1/4} (\mathfrak{e}_{(1/2,0)} + \mathfrak{e}_{(0,1/2)}) \end{aligned}$ \\ \hline
$6$ & $2$ & $\begin{aligned} &4 \mathfrak{e}_{(0,0)} + 2q^{-1/24} (\mathfrak{e}_{(0,1/12)} + \mathfrak{e}_{(0,5/12)} + \mathfrak{e}_{(0,7/12)} + \mathfrak{e}_{(0,11/12)}) + \\ &\quad + q^{-1/6} (\mathfrak{e}_{(0,1/6)} + \mathfrak{e}_{(0,5/6)}) + q^{-1/4} \mathfrak{e}_{(1/2,0)} \end{aligned}$ \\ \hline
$10$ & $3$ & $\begin{aligned} &6 \mathfrak{e}_{(0,0)} + q^{-1/40} (\mathfrak{e}_{(0,1/20)} + \mathfrak{e}_{(0,9/20)} + \mathfrak{e}_{(0,11/20)} + \mathfrak{e}_{(0,19/20)}) + \\ &\quad + q^{-3/20} (\mathfrak{e}_{(1/2,3/10)} + \mathfrak{e}_{(1/2,7/10)}) + \\ &\quad + q^{-9/40} (\mathfrak{e}_{(0,3/20)} + \mathfrak{e}_{(0,7/20)} + \mathfrak{e}_{(0,13/20)} + \mathfrak{e}_{(0,17/20)}) + \\ &\quad + q^{-1/4} \mathfrak{e}_{(1/2,0)} \end{aligned}$ \\ \hline
$13$ & $3$ & $\begin{aligned} &6 \mathfrak{e}_{(0,0)} + q^{-1/13} (\mathfrak{e}_{(0,1/13)} + \mathfrak{e}_{(0,12/13)}) + \\ &\quad + q^{-9/52} (\mathfrak{e}_{(0,3/26)} + \mathfrak{e}_{(0,23/26)} + \mathfrak{e}_{(1/2,5/13)} + \mathfrak{e}_{(1/2,8/13)}) + \\ &\quad + q^{-3/13} (\mathfrak{e}_{(0,4/13)} + \mathfrak{e}_{(0,9/13)}) + \\ &\quad + q^{-1/4} (\mathfrak{e}_{(0,1/2)} + \mathfrak{e}_{(1/2,0)}) \end{aligned}$ \\ \hline
$14$ & $2$ & $\begin{aligned} &4 \mathfrak{e}_{(0,0)} + q^{-1/56} (\mathfrak{e}_{(0,1/28)} + \mathfrak{e}_{(0,13/28)} + \mathfrak{e}_{(0,15/28)} + \mathfrak{e}_{(0,27/28)}) + \\ &\quad + q^{-1/28} (\mathfrak{e}_{(1/2,5/14)} + \mathfrak{e}_{(1/2,9/14)}) + \\ &\quad + q^{-1/14}(\mathfrak{e}_{(0,1/14)} + \mathfrak{e}_{(0,13/14)}) + \\ &\quad + q^{-1/7} (\mathfrak{e}_{(0,2/7)} + \mathfrak{e}_{(0,5/7)}) + \\ &\quad + q^{-9/56} (\mathfrak{e}_{(0,3/28)} + \mathfrak{e}_{(0,11/28)} + \mathfrak{e}_{(0,17/28)} + \mathfrak{e}_{(0,25/28)}) + \\ &\quad + q^{-1/4} \mathfrak{e}_{(1/2,0)} \end{aligned}$ \\ \hline
$17$ & $2$ & $\begin{aligned} &4 \mathfrak{e}_{(0,0)} + 2q^{-1/17} (\mathfrak{e}_{(0,1/17)} + \mathfrak{e}_{(0,6/17)} + \mathfrak{e}_{(0,11/17)} + \mathfrak{e}_{(0,16/17)}) + \\ &\quad + q^{-9/68} (\mathfrak{e}_{(0,3/34)} + \mathfrak{e}_{(0,31/34)} + \mathfrak{e}_{(1/2,7/17)} + \mathfrak{e}_{(1/2,10/17)}) + \\ &\quad + q^{-1/4} (\mathfrak{e}_{(1/2,0)} + \mathfrak{e}_{(0,1/2)}) \end{aligned}$ \\ \hline
$17$ & $3$ & $\begin{aligned} &6 \mathfrak{e}_{(0,0)} + q^{-1/68} (\mathfrak{e}_{(0,1/34)} + \mathfrak{e}_{(0,33/34)} + \mathfrak{e}_{(1/2,8/17)} + \mathfrak{e}_{(1/2,9/17)}) + \\ &\quad + q^{-2/17} (\mathfrak{e}_{(0,6/17)} + \mathfrak{e}_{(0,11/17)}) + \\ &\quad + q^{-4/17} (\mathfrak{e}_{(0,2/17)} + \mathfrak{e}_{(0,15/17)}) + \\ &\quad + q^{-13/68} (\mathfrak{e}_{(0,9/34)} + \mathfrak{e}_{(0,25/34)} + \mathfrak{e}_{(1/2,4/17)} + \mathfrak{e}_{(1/2,13/17)}) \\ &\quad + q^{-1/4} (\mathfrak{e}_{(1/2,0)} + \mathfrak{e}_{(0,1/2)}) \end{aligned}$ \\ \hline
\end{longtable}

\end{ex}

\textbf{Acknowledgments:} I thank Shaul Zemel for pointing out some errors in the published version of this note. The author is supported by the LOEWE research unit Uniformized Structures in Arithmetic and Geometry. 

\bibliographystyle{plainnat}
\bibliography{\jobname}

\end{document}